\documentclass[referee,envcountsect,reqno,12pt]{svjour3}
\usepackage{amsmath,amssymb,upgreek,lipsum}
\usepackage{fontenc,enumerate,amsfonts,hyperref,mathrsfs,amsfonts,dsfont,cite}
\usepackage{fancyhdr}
\smartqed
\hypersetup{colorlinks=true,linkcolor=red, anchorcolor=blue, citecolor=blue, urlcolor=red, filecolor=magenta, pdftoolbar=true}
\makeatletter
\newcommand{\rom}[1]{\expandafter\@slowromancap\romannumeral #1@}
\makeatother
\numberwithin{equation}{section}
\usepackage{enumerate}
\usepackage{graphicx}
\allowdisplaybreaks
\begin{document}
\title{Existence of Solutions for Non-Autonomous Second-Order Stochastic Inclusions with Clarke's Subdifferential and non Instantaneous Impulses}
\author{Anjali Upadhyay\and Surendra Kumar  }

\institute{Anjali Upadhyay \at  Faculty of Mathematical Sciences, Department of Mathematics, University of Delhi, New Delhi, 110007\\
      \email{upadhyayanjali123@gmail.com} \and Surendra Kumar \at
              Faculty of Mathematical Sciences, Department of Mathematics, University of Delhi, New Delhi, 110007 \\
              \email{mathdma@gmail.com} }
\maketitle
\begin{abstract}
This manuscript explores a new class of non-autonomous second-order stochastic inclusions of Clarke's subdifferential form with non-instantaneous impulses (NIIs), unbounded delay and the Rosenblatt process in Hilbert spaces. The existence of a solution is deduced by employing a fixed point strategy for a set-valued map together with the evolution operator and stochastic analysis approach.  An example is analyzed for theoretical developments.
\end{abstract}
\keywords{Second-order non-autonomous system \and stochastic evolution inclusions \and Fixed point \and Rosenblatt process  \and Infinite delay \and Clarke`s generalized subdifferential \and Mild solution}
\subclass{34A12 \and 34A37 \and 34A60 \and 35K30 \and 60H10}

\baselineskip12pt
\section{Introduction}
 Due to its practical applications in several fields, for instance, finance, physics, electrical engineering, medicine, and telecommunication, among others, many scholars have addressed stochastic evolution equations, and have already gained several fruitful results.
 To explore more, we refer to the following books and articles and the references cited therein \cite{Stochastic,Mao,Prato1992,Ren&Sun}.
  In numerous regions of research, there has been a growing revenue in the evaluation of the frameworks fusing memory, i.e.,
 there is the impact of postponement on state conditions. In this way, there is a genuine desire to talk about stochastic differential systems with delay.

In abstract spaces, Henr\'{i}quez \cite{Henriquez2011} assessed the presence of mild solutions as well as classical solutions
for a non-autonomous second-order (NASO) delayed functional differential equation with unbounded delay.
Henr\'{i}quez et~al. \cite{Henriquezetal2014} considered NASO differential structure with nonlocal initial data and developed the existence of solutions by applying the principle of the Hausdorff measure of non-compactness.
   Benchohra et~al. \cite{Benchohraetal2019} used a fixed point theorem developed by Darbo with the Kuratowski measure of non-compactness to build certain adequate conditions that guarantee the presence of a solution for a NASO non-instantaneous integro differential system.

 Because of its easy calculus and interesting attributes, the fractional Brownian motion (fBm) has attracted to many scholars. One can go through the works in \cite{Prato1992, BS,CG} for further details. In certain cases, where the mechanism is not Gaussian, the Rosenblatt process is chosen over fBm.  Although the theory of the Rosenblatt process was established during 60's and 70's, significant development has been made in the last decade because of its self-similarity, long-range dependency, stationary increments.

 In the literature, there are numerous articles which studied various theoretical facets of the Rosenblatt process.
 Ouahra et~al. \cite{Ouahraetal2017} discussed the qualitative properties of an stochastic  delayed neutral functional differential system  including impulses, Poisson jump and the Rosenblatt process.
 Leonenko and Ahn \cite{NV} gave a fruitful result for the rate of convergence of the Rosenblatt process.
 The distribution property of the Rosenblatt process was investigated by Maejma and Tudor \cite{MC}.
 Sakthivel et~al.\cite{Sakthiveletal2018}  analysed an abstract NASO stochastic evolution model with unbounded delay governed by the Rosenblatt process and proved its existence using the Krasnoselskii—Schaefer fixed point theorem, also studied the related autonomous system with bounded delay.

 Lakhel and Tlidi \cite{Lakhel&Tlidi2019} employed the Banach fixed point theorem to
discuss the existence, uniqueness, and established stability criteria for a neutral stochastic
functional differential system with impulses involving variable delays governed by the Rosenblatt process.

On the other hand,  Clarke's subdifferential emerges from the applied discipline,
 namely  thermo-viscoelasticity, filtration in porous materials, riveting applications in optimization
 and non-smooth analysis \cite{1,Lu&Bin}.
 Recently, Vijayakumar \cite{VV} considered NASO stochastic inclusions of
 Clarke's subdifferential form and established the approximate controllability for the proposed systems.

 Hernandez and O'Regan \cite{Hernandez&Regan} introduced the theory of NIIs.
Thereafter, many researchers gave various results on differential equations with NIIs.
Pierri et~al. \cite{Pierri&O'Regan}, Yu and Wang \cite{Yu&Wang}, Fekan and Wang \cite{Michal&Wang}, many more \cite{Yan&Jia,Malik&Kumar} studied various qualitative
properties of differential systems with NIIs. To the best of our incite, no result guarantees
the existence of a solution for a NASO stochastic  differential inclusions with Clarke's subdifferential including the Rosenblatt process and NIIs.

Let $\mathfrak{Z}$ and $(\mathscr{Z},\|\cdot\|_\mathscr{Z}, \langle\cdot\rangle_\mathscr{Z})$ be Hilbert spaces that are real and separable. The notation $L(\mathscr{Z},\mathfrak{Z})$ reflects the space of all bounded linear operators from $\mathscr{Z}$ into $\mathfrak{Z}$. Strongly motivated by the above facts and discussions, we examine the subsequent stochastic differential inclusion with unbounded delay and NIIs
\begin{align}\label{mainequation}
\left\{
  \begin{array}{ll}
   d\chi'(\tau) \in [ \mathscr{A}(\tau)\chi(\tau)+ \partial\Sigma(\tau,\chi(\tau))]d\tau +q(\tau,\chi_\tau)dZ_H(\tau), &  \tau\in \mathop{\cup}\limits_{k=0}^\mathcal{M}(t_k,r_{k+1}]; \\
    \chi(\tau)= f_k(\tau,\chi_\tau),& \tau\in \mathop{\cup}\limits_{k=1}^\mathcal{M}(r_k,t_k];\\
\chi'(\tau)= g_k(\tau,\chi_\tau),& \tau\in \mathop{\cup}\limits_{k=1}^\mathcal{M}(r_k,t_k];\\
\chi(\tau)= \eta(\tau), & \tau\in ]-\infty,0];\\
\chi'(0)=\xi,
  \end{array}
\right.
\end{align}
where, $\chi(\cdot)$ is $\mathfrak{Z}$-valued stochastic process, $\mathscr{A}(\tau): D(\mathscr A(\tau))\subseteq \mathfrak{Z} \rightarrow \mathfrak{Z}$ is  closed and linear whose domain is dense in $\mathfrak{Z}$. The $\tau-$segment of $\chi$, $\chi_\tau :]- \infty,0] \rightarrow \mathfrak{Z}$ is given by  $\chi_\tau(\theta) = \chi(\tau+\theta);\ \theta \in]-\infty,0]$, and belongs to an abstract phase space $\mathscr{W}$ described in Sect. \ref{Sect.2}. Let $\mathfrak{J} = [0,\beta],\  \mathfrak{J}_0=]-\infty,0]$. The notation $\partial \Sigma$ represents the Clarke generalized subdifferential (see \cite{1}) of a locally Lipschitz function $\Sigma(\tau,\cdot): \mathfrak{Z} \rightarrow \mathds{R}$; $q:\mathfrak{J} \times \mathscr{W} \rightarrow  L_2^0,\ f_k,g_k :(r_k,t_k] \times \mathscr{W} \rightarrow \mathfrak{Z},\ k=1,\hdots, \mathcal{M}$ are suitable functions.  The initial data $\eta$ is $\Gamma_0$-measurable $\mathscr{W}$-valued stochastic process and $\xi$ is $\Gamma_0$-measurable $\mathfrak{Z}$-valued stochastic process. Also $\eta$ and $\xi$ have finite second moment, and are independent of the Rosenblatt process $Z_H$.

 The points $0=r_0<t_0<r_1 <t_1 < \hdots <t_\mathcal{M} <r_{\mathcal{M}+1}=\beta$ are impulsive positions. The impulses begin abruptly at $r_k$ and continue to have an impact on $(r_k,t_k]$. The function $\chi(\cdot)$ takes distinct values in the two subintervals $(r_k,t_k], \ (t_k,r_{k+1}]$ and is continuous at $t_k$.

The following is the summary of the rest of the manuscript: Sect. \ref{Sect.2} is devoted to basic results, concepts, and Lemmas. Existence result  for the proposed system (\ref{mainequation}) is covered in Sect. \ref{Sect.3}  by  using set-valued(multi-valued) fixed point theorem \cite{2}. We have reserved Sect. \ref{Sect.4} for an example to show the applicability of the acquired result.
\section{Preliminaries}\label{Sect.2}
 Consider the probability space $(\Omega,\Gamma,\{\Gamma_\tau\}_{\tau\geq 0},\mathbb{P})$ that is complete with the
 right continuous increasing sub $\sigma$-algebras $\{\Gamma_\tau\}_{\tau \in \mathfrak{J}}$ with $\Gamma_\tau \in \Gamma$
  generated by all $\mathbb{P}$-null sets and $\{Z_H(t), t \in [0,\tau]\}$; where
$Z_H(\tau)$ represents the Rosenblatt process on  $\mathscr{Z}$, and $H\in (\frac{1}{2},1)$ .
Let $L_2(\Omega,\mathfrak{Z})$ be the Banach space of strongly measurable, $\mathfrak{Z}$-valued random variables
with the norm $\|\chi(\cdot)\|_{L_2}=(\mathbb{E}\|\chi(\cdot)\|^2_\mathfrak{Z})^\frac{1}{2}$,
where $\mathbb{E}\|\chi\|^2=\int_\Omega \|\chi\|^2 d\mathbb{P}<\infty$. Let  $\{ e_i\}_{i=1}^{\infty}$ be a complete orthonormal basis for $\mathscr{Z}$.
The operator $Q\in L(\mathscr{Z})$ is defined by $Qe_i=\nu _ie_i , $ $i\in \mathds{N}$ with trace
$Tr(Q)=\mathop  {\sum}\limits_{i=1}^{\infty}\nu_i<\infty;$ $\nu_i\geq 0$. Denote a sequence of mutually independent Rosenblatt processes by $\{z_i(\tau)\}_{i}^{\infty}$ on $(\Omega,\Gamma,\mathbb{P})$, which are two-sided and one-dimensional. A $\mathscr{Z}$-valued stochastic
process $Z_Q(\ell)$ is defined as
\begin{align*}
Z_Q(\ell)=\mathop {\sum}\limits_{i=1}^{\infty}z_i(\ell)Q^\frac{1}{2}e_i,~ \ell \geq 0.
\end{align*}
Moreover, the above series is convergent in $\mathscr{Z}$ if $Q \geq 0$ and $Q=Q^*$(adjoint of $Q$).

Let $\mathscr{Z}_0=Q^\frac{1}{2}\mathscr{Z}$ be the Hilbert space with the inner product
$\langle z_1,z_2\rangle_{\mathscr{Z}_0}=\langle Q^\frac{1}{2}z_1, Q^\frac{1}{2} z_2 \rangle_\mathscr{Z}$.
Further, let $L_2(\mathscr{Z}_0,\mathfrak{Z}):=L_2^0$ be the space of Hilbert--Schmidt operators from $\mathscr{Z}_0$ into $\mathfrak{Z}$.
Clearly $L_2^0$ equipped with the inner product $\langle\Phi_1,\Phi_2\rangle=\mathop {\sum}\limits_{i=1}^{\infty} \langle\Phi_1e_i,\Phi_2 e_i\rangle$
is a Hilbert space. Moreover, $\|\Phi\|_{L^2_0}^2=\|\Phi Q^{\frac{1}{2}}\|^2=Tr(\Phi Q \Phi^*)$.

Let $\mathfrak{J}$ be the interval with given horizon $\beta$. Then the one dimensional Rosenblatt process is represented by \cite{3}
\begin{align*}\label{3}
Z^{\alpha}_H(\tau)= c(H)\int_0^\tau \int_0^\tau \Big[\int_{b_1 \vee b_2}^\tau \frac{\partial K^{\hat{H}}}{\partial u}(u,b_1)\frac{\partial K^{\hat{H}}}{\partial u}(u,b_2) ds \Big]dB(b_1)dB(b_2)
\end{align*}
where $B=\{B(\tau): \tau \in \mathfrak{J}\}$ is the Wiener process,
$\hat{H}=\frac{H+1}{2}, \ c(H)=\frac{1}{H+1} \sqrt{\frac{H}{2(2H-1)}}$ and the kernel $K^H(\cdot,\cdot)$ is  given by
\begin{align*}
K^H(\ell,s)=
\begin{cases}
c_H \ s^{\frac{1}{2}-H} \int_s^\ell (u-s)^{H-\frac{3}{2}}u^{H-\frac{1}{2}}du,\ \ell >s;\\
0,\hspace{4.5cm} \ell \leq s
\end{cases}
\end{align*}
 where $c_H=\sqrt{\frac{H(2H-1)}{\mathfrak{B}(2-2H,H-\frac{1}{2})}};$ $\mathfrak{B}(\cdot,\cdot)$ represents the Beta function.
%The covariance of the Rosenblatt process $\{Z_H(\ell), \ell \in\mathfrak{J}\}$ satisfies following relation
%\begin{align*}
%R_H(s,\ell):= \mathbb{E}(Z_H(\ell)Z_H(s))=\frac{1}{2}(\ell^2H+s^2H-|\ell-s|^2H).
%\end{align*}
The space $PC(\mathfrak{Z})$ formed by all $\mathfrak{Z}$-valued stochastic processes $\{\chi(\tau) : \tau\in\mathfrak{J}\}$ that are $\Gamma_\tau$-adapted measurable, with $\chi$ is continuous at $\tau \neq r_k, \  \chi_{r_k}= \chi_{r_k^- },$ and $ \chi_{r_k^+ }$ exist for $k = 1,...,\mathcal{M}$ is a Banach space with $\|\chi\|_{PC} = \big(\mathop{\sup}\limits _{\beta \geq s \geq 0} \ \mathbb{E}\|\chi(s)\|^2\big)^\frac{1}{2}$.

The phase space $(\mathscr{W}, \|\cdot\|_{\mathscr{W}})$ formed by all $\Gamma_0$-measurable mappings from
$\mathfrak{J}_0$ into $\mathfrak{Z}$ is a semi-normed linear space and the accompanying axioms hold (cf.\cite{Hino,Hale&Kato1978})
\begin{enumerate}[(i)]
\item
If $\chi: \ ]-\infty, \beta] \rightarrow \mathfrak{Z},$ $\beta>0$, is such that $\chi|_{[0,\beta]} \in  PC([0,\beta],\mathfrak{Z})$ with $\chi_0 \in \mathscr{W}$, then for all $\tau \in [0,\beta]$ following  hold:
\begin{enumerate}
\item $\chi_\tau \in \mathscr{W}$; \
 \item $\|\chi(\tau)\| \leq J \|\chi_\tau\|_{\mathscr{W}}$, where $J>0$ is a constant;\
 \item $\|\chi_\tau\|_{\mathscr{W}} \leq \  K(\tau) \sup \{\|\chi(s)\|:\ \tau \geq s \geq 0 \} + L(\tau)\|\chi_0\|_\mathscr{W}$,\\
  where $K, L : \mathds{R}^+\cup\{0\} \rightarrow [1,\infty)$, $K$  and $L$ are continuous and  locally bounded respectively, and independent of $\chi(\cdot)$.
\end{enumerate}
\item
  $\mathscr{W}$ is complete space.
\end{enumerate}
The result given below is extracted from the above axioms:
\begin{lemma} \label{lemma2.1}\cite{Yan&Jia}
Let the process $\chi : ]-\infty, \beta] \rightarrow \mathfrak{Z}$  be measurable and $\Gamma_\tau$-adapted with \\$\chi|_{\mathfrak{J}} \in PC(\mathfrak{Z})$, $\chi_0 = \eta(\tau) \in \mathscr{L}_{\Gamma_0}^2(\Omega, \mathscr{W})$, then
\begin{center}
$\|\chi_\tau\|_{\mathscr{W}} \leq \ K_\beta\mathop{\sup}\limits_{\tau \in  \mathfrak{J}} \mathbb{E}\|\chi(\tau)\|+L_\beta\|\eta\|_{\mathscr{W}}$,
\end{center}
where $ K_\beta = \mathop{\max}\limits_{\tau \in \mathfrak{J}} K(\tau)$ and $ L_\beta =  \mathop{\sup}\limits_{\tau \in \mathfrak{J}} L(\tau)$.
\end{lemma}

For further development, we describe the following theory of evolution operator.
\begin{definition}\cite{Kozak1995}
A mapping $\mathcal{G} :\mathfrak{J} \times \mathfrak{J} \rightarrow L(\mathfrak{Z})$ is characterized as an evolution operator for $\chi''(\tau)= \mathscr{A}(\tau)\chi(\tau),\ \ \beta \geq s,\ \tau \geq 0$ if meet the
subsequent conditions:
\begin{enumerate}
\item[$(B_1)$]  The map $(\tau,s) \mapsto \mathcal{G}(\tau,s)\chi$ is of class $C^1$ for every $\chi \in \mathfrak{Z}$, and
\begin{enumerate}[(i)]
  \item  $\mathcal{G}(\tau,\tau)=0$ for every $\tau \in \mathfrak{J}$,
  \item  For every $\chi \in \mathfrak{Z}$, for all $\tau, s \in \mathfrak{J},$
  \begin{align*}
\frac{\partial}{\partial \tau} \mathcal{G}(\tau,s)|_{\tau=s}\  \chi =\chi~\mbox{and}~ \frac{\partial}{\partial s} \mathcal{G}(\tau,s)|_{\tau=s} \ \chi = -\chi.
\end{align*}
\end{enumerate}
\item[$(B_2)$] For  $\chi \in D(A(\tau))$, $\mathcal{G}(\tau,s)\chi \in D(A(\tau))$ for all $s, \tau \in \mathfrak{J}$, the map $(\tau,s) \mapsto \mathcal{G}(\tau,s)\chi$ is of class $C^2$ and
    \begin{enumerate}[(i)]
      \item $\frac{\partial^2}{\partial \tau^2} \mathcal{G}(\tau,s)\chi=A(\tau)\mathcal{G}(\tau,s)\chi$,
      \item $\frac{\partial^2}{\partial s^2} \mathcal{G}(\tau,s)\chi=\mathcal{G}(\tau,s)A(s)\chi$,
      \item $\frac{\partial^2}{\partial s \partial \tau} \mathcal{G}(\tau,s)|_{\tau=s} \ \chi=0$
    \end{enumerate}
\item[$(B_3)$] For all $s, \tau \in \mathfrak{J},$ if $\chi \in D(A(\tau))$ then $\frac{\partial}{\partial s} \mathcal{G}(\tau,s)\chi \in D(A(\tau))$,  $\frac{\partial^3}{\partial \tau^2 \partial s} \mathcal{G}(\tau,s)\chi,\ \frac{\partial^3}{\partial s^2 \partial \tau } \mathcal{G}(\tau,s)\chi$ exist, also
    \begin{enumerate}[(i)]
      \item $\frac{\partial^3}{\partial \tau^2 \partial s} \mathcal{G}(\tau,s)\chi=A(\tau)\frac{\partial}{\partial s}\mathcal{G}(\tau,s)\chi,$
      \item $\frac{\partial^3}{\partial s^2 \partial \tau } \mathcal{G}(\tau,s)\chi= \frac{\partial}{\partial \tau} \mathcal{G}(\tau,s)A(s)\chi,$
    \end{enumerate}
and the map $(\tau,s) \mapsto A(\tau)\frac{\partial}{\partial s}\mathcal{G}(\tau,s)\chi$ is continuous.
\end{enumerate}
\end{definition}
Throughout the article, we suppose an evolution operator $\mathcal{G}(\tau,s)$ exists related to the operator $A(\tau)$.
Besides, we present $\mathcal{E}(\tau,s)=-\frac{\partial}{\partial s} \mathcal{G}(\tau,s)$.

We are now presenting some useful definitions for the set-valued map (see \cite{6,7}).
Let $P(\mathfrak{Z})$ denote the family of all non-empty subsets of $\mathfrak{Z}$. For convenience, set:
\begin{align*}
P_{cl}(\mathfrak{Z})=\{\chi \in P(\mathfrak{Z}): \chi \  \mbox{is closed}\},~ P_{bd}(\mathfrak{Z})=\{\chi \in P(\mathfrak{Z}): \chi \ \mbox{is bounded}\},\\
P_{cv}(\mathfrak{Z})=\{\chi \in P(\mathfrak{Z}): \chi \  \mbox{is convex}\},~ P_{cp}(\mathfrak{Z})=\{\chi \in P(\mathfrak{Z}): \chi \ \mbox{is compact}\},
\end{align*}
Consider $\mathfrak{Z}_d : P(\mathfrak{Z})\times P(\mathfrak{Z}) \rightarrow \mathds{R}^+ \cup \{\infty\}$ given  by
\begin{align*}
\mathfrak{Z}_d(\mathds{G},\mathds{H})=\max \big\{\mathop{\sup}\limits_{u \in \mathds{G}} d(u,\mathds{H}),\ \mathop{\sup}\limits_{v \in \mathds{H}} d(\mathds{G},v)\big\},
\end{align*}
where $ d(u,\mathds{H})=\mathop{\inf}\limits_{v \in \mathds{H}}d(u,v),\ d(\mathds{G},v)= \mathop{\inf}\limits_{u \in \mathds{G}}d(u,v).$ Then $(P_{bd,cl}(\mathfrak{Z}),\mathfrak{Z}_d)$ is a metric space.
\begin{definition}
Let $\Theta : \mathfrak{Z} \rightarrow P(\mathfrak{Z})$ be a set-valued mapping, then
\begin{enumerate}[(i)]
\item
$\Theta$ is  closed (convex) valued if $\Theta(\chi)$ is closed (convex) for every $\chi \in \mathfrak{Z}$.
\item
 $\Theta$ is bounded on bounded sets if $\Theta (\mathscr{D})=\cup_{\chi \in \mathscr{D}}$ $ \Theta (\chi)$ is bounded in $\mathfrak{Z}$ for
 all $\mathscr{D}\in P_{bd}(\mathfrak{Z})$.
\item
 If for each $\chi \in \mathfrak{Z}$,  $\Theta(\chi)\neq\emptyset$ is closed subset of $\mathfrak{Z},$ and if for each open set $J$ in $\mathfrak{Z}$  containing $\Theta(\chi)$, there is an open neighbourhood $O$ of $\chi$ such that $\Theta(O)\subseteq J$, then $\Theta$ is characterised as upper semi continuous (u.s.c.) on $\mathfrak{Z}$, 
\item
 If $\Theta(J)$ is relatively compact for every $J \in P_{bd}(\mathfrak{Z})$, then $\Theta$ is completely continuous .
%\item
%If $\Theta$ is completely continuous with nonempty compact values, then $\Theta$ is u.s.c. if and only if $\Theta$ has a closed graph.\
 \item
If there is a $\chi \in \mathfrak{Z}$ such that $\chi \in \Theta(\chi)$, then $\Theta$ has a fixed element.
\end{enumerate}
\end{definition}
\begin{definition}
A set-valued operator $\Theta: \mathfrak{Z} \to P_{bd,cl} (\mathfrak{Z})$ is known to be contraction if there is \\$\gamma \in (0,1)$ to ensure that
\begin{center}
$\mathfrak{Z}_d(\chi_1,\chi_2)  \leq \gamma \ d(\chi_1,\chi_2),\ \forall \  \chi_1,\chi_2 \in \mathfrak{Z}.$
\end{center}

\end{definition}
\begin{definition}

The Clarke generalized directional derivative of a locally Lipschitz functional  \\$\Sigma : \mathfrak{Z} \rightarrow \mathds{R}$ at $\mathrm{z} \in \mathfrak{Z}$ in the direction $w$  is defined as 
\begin{center}
$\Sigma^0 (\mathrm{z} ;w ) =  \mathop{\limsup}\limits_{x \rightarrow \mathrm{z} \  \varepsilon \rightarrow 0^+} \frac{\Sigma(x + \varepsilon w )-\Sigma (x )} {\varepsilon}$.\end{center}
\vspace{.2cm}
The Clarke generalized subdifferential of $\Sigma$ is a subset of $\mathfrak{Z}^*$, and at a point $\mathrm{z} \in \mathfrak{Z}$ is defined as 
  \begin{center}
  $\partial \Sigma (\mathrm{z}) = \{ \mathrm{z}^* \in \mathfrak{Z}^* : \Sigma^0 (\mathrm{z} ;w ) \geq \langle \mathrm{z}^*, w \rangle$, for all $w \in  \mathfrak{Z}\}$
  \end{center}
\end{definition}
\begin{lemma}\label{lemma2.2}\cite{2}
Let ${\bar\Theta}_1:\mathfrak{Z} \rightarrow P_{cl,cv,bd}(\mathfrak{Z})$, ${\bar\Theta}_2:\mathfrak{Z}\rightarrow P_{cl,cv}(\mathfrak{Z})$
be set-valued maps satisfying
\begin{enumerate}
\item[(a)] ${\bar\Theta}_1$ is a contraction,
\item[(b)] ${\bar\Theta}_2$ is u.s.c. and completely continuous.
\end{enumerate}
Then either $(i)$ the inclusion $\lambda \chi \in {\bar\Theta}_1\chi+{\bar\Theta}_2 \chi$ has a solution for $\lambda=1$, or\\
$(ii)$ the set $\{\chi \in \mathfrak{Z}: \lambda \chi \in {\bar\Theta}_1\chi+{\bar\Theta}_2 \chi,\ \lambda > 1\}$ is unbounded.
\end{lemma}

The following result play a key role in dealing with stochastic term:
\begin{lemma} \label{lemma2.3}\cite{8}
Let $\phi:\mathfrak{J}\rightarrow L_2^0$ be such that $\mathop {\sup}\limits_{\tau \in \mathfrak{J}}\|\phi\|_{L_2^0}^2<\infty$.
 Suppose that there is $M>0$ satisfying $\|\mathcal{G}(\tau,s)\|^2\leq M$ for all $\tau \geq s$. Then
 \begin{center}$\mathbb{E}\|\int_0^\tau \mathcal{G}(\tau,s)\phi(s)dZ_H(s)\|^2_\mathfrak{Z} \leq c(H)M \tau^{2H}\big(\mathop{\sup}\limits_{\tau \in \mathfrak{J}} \|\phi\|_{L_2^0}^2\big)$.
 \end{center}
%Hence, the stochastic integral $\int_0^\tau \mathcal{G}(\tau,s)\phi(s) dZ_H(s)$ is well defined.
 \end{lemma}
 Now we introduce a solution of proposed system (\ref{mainequation}) as follows
\begin{definition}
 An stochastic process $\chi : ]-\infty, \beta] \rightarrow \mathfrak{Z}$ is called a mild solution for (\ref{mainequation}) if
\begin{enumerate}
\item the measurable process $\chi_\tau$ is adapted to $\Gamma_\tau,$ $ \tau \geq 0$,
%\item $\chi(\tau) \in \mathfrak{Z}$ has $c\acute{a}dl\acute{a}g$ path on $\tau\in\mathfrak{J}$ a.s.,
\item $\ \chi= \eta(\tau)$ on $]-\infty,0]$, satisfying $\|\eta\|_\mathscr{W}^2<\infty ,\ \chi_\tau \in \mathscr{W}$, $\tau \in \mathfrak{J}$ with $\chi'(0)=\xi \in \mathfrak{Z},\ \chi|_{\mathfrak{J}} \in PC(\mathfrak{Z})$ and following integral equation hold:
\begin{equation}\label{mildsolution}
\chi(\tau)=
\begin{cases}
%\begin{aligned}
 \mathcal{E}(\tau,0)\eta(0)+\mathcal{G}(\tau,0)\xi+ \int_0^\tau \mathcal{G} (\tau,s) \rho(s)ds\\+ \int_0^\tau \mathcal{G}(\tau,s)q(s,\chi_s)dZ_H(s),\hspace{3.9cm} \tau \in [0,r_1];\\
 f_k(\tau, \chi_\tau),\hspace{6cm} \tau\in \mathop{\cup}\limits_{k=1}^{\mathcal{M}}(r_k,t_k]; \\
 \mathcal{E}(\tau,t_k)f_k(t_k,\chi_{t_k})+ \mathcal{G}(\tau,t_k)g_k(t_k,\chi_{t_k})
+ \int_{t_k}^\tau  \mathcal{G} (\tau,s) \rho(s)ds \\+ \int_{t_k}^\tau \mathcal{G} (\tau,s) q(s, \chi_s)dZ_H(s),\hspace{3cm} \tau\in \mathop{\cup}\limits_{k=1}^{\mathcal{M}}(t_k,r_{k+1}].
%\end{aligned}
\end{cases}
\end{equation}
\end{enumerate}
\end{definition}
\section{Existence of solution}\label{Sect.3}
 We start this section by imposing the following conditions on the system parameters:
\begin{itemize}
\item [$(\textbf{S}_1):$] 
The operators $\mathcal{E}(\tau,s)$ and $\mathcal{G}(\tau,s)$ are compact for all $\tau \geq s$, and there exists $\mathrm{M}>0$ such that
\begin{align*}
\mathop{\sup}\limits_{(\tau,s)\in \mathfrak{J}\times \mathfrak{J}}\|\mathcal{E}(\tau,s)\| \ \vee \ \mathop{\sup}\limits_{(\tau,s) \in \mathfrak{J} \times \mathfrak{J}}\|\mathcal{G}(\tau,s)\|\ \leq \mathrm{M}, \ \mbox{for all}\  \tau \geq s.
\end{align*}
\item [$(\textbf{S}_2):$] The functions $f_k, g_k :(r_k,t_k]\times\mathscr{W} \rightarrow \mathfrak{Z}$ are continuous, and also there are $c_k>0,\ \gamma_k>0,\ k=1,\hdots \mathcal{M}$ in order that for all $\eta, \eta_1,\eta_2 \in \mathscr{W}$,
\begin{align*}
\mathbb{E}\|f_k(\tau,\eta_1)-f_k(\tau,\eta_2)\|^2_\mathfrak{Z}\ \leq \ \gamma_k \|\eta_1-\eta_2\|_{\mathscr{W}}^2, \
\mathbb{E}\|f_k(\tau,\eta)\|^2\ \leq \ \gamma_k(1+\|\eta\|_{\mathscr{W}}^2), \\
\mathbb{E}\|g_k(\tau,\eta_1)-g_k(\tau,\eta_2)\|^2_\mathfrak{Z}\ \leq \ c_k \|\eta_1-\eta_2\|_{\mathscr{W}}^2, \
\mathbb{E}\|g_k(\tau,\eta)\|^2\ \leq \ c_k(1+\|\eta\|_{\mathscr{W}}^2).
\end{align*}
\item[$(\textbf{S}_3):$]  Let $\Sigma: \mathfrak{J} \times \mathscr{W} \rightarrow \mathds{R}$ be the map such that:
\item[(i)] For all $\chi \in  \mathfrak{Z},\ \Sigma(\cdot,\chi)$ is measurable .
\item[(ii)] For a.e. $\tau \in \mathfrak{J},\ \Sigma(\tau,\cdot)$ is locally Lipschitz .
\item[(iii)] There is $b_1(\cdot) \in L^1(\mathfrak{J},\mathds{R}^+)\ \mbox{and}\ 0 \leq b_2 $ in order that
\begin{align*}
\|\partial \Sigma (s,\chi)\|^2&=\sup\{\|\rho(s)\|^2|\ \rho(s)\in \partial \Sigma (s,\chi)\}\\
&\leq b_1(s)+b_2\|\chi\|^2\ \mbox{for all}\ \chi \in \mathfrak{Z},\  \mbox{\ a.e.} \ s \in \mathfrak{J}.
\end{align*}
\item[$(\textbf{S}_4):$]$(i)$ The function $q(\tau,\cdot):\mathscr{W} \rightarrow L_2^0$ is continuous for all $\tau \in \mathfrak{J}$ and  $q(\cdot,\eta):\mathfrak{J} \rightarrow L_2^0$ is strongly measurable for each $\eta \in \mathscr{W}$. Also there is $ M_q>0$ to ensure that
 \begin{align*}
 \mathbb{E}\|q(\tau,\eta_1)-q(\tau,\eta_2)\|^2_{L_2^0}\ \leq \ M_q \|\eta_1-\eta_2\|_{\mathscr{W}}^2,\   \ \eta_1,\eta_2 \in \mathscr{W}.
\end{align*}
\item[(ii)] There is a continuous function $m_q:[0,\infty) \rightarrow (0,\infty)$ that is non decreasing, and $m(\cdot)\in L^1(\mathfrak{J},\mathds{R}^+)$ with the aim that
\begin{align*}
 \mathbb{E}\|q(\tau,\eta)\|^2_{L_2^0}\ \leq \ m(\tau)m_q (\|\eta\|_{\mathscr{W}}^2), \ \  (\tau,\eta)\in \mathfrak{J}\times \mathscr{W}.
 \end{align*}
\end{itemize}
Consider the set-valued map $S:\mathscr{L}^2(\mathfrak{J},\mathfrak{Z})\to 2^{\mathscr{L}^2(\mathfrak{J},\mathfrak{Z})}$given by
\begin{align*}
S_{\Sigma,\chi}=\{\rho\in \mathscr{L}^2(\mathfrak{J},\mathfrak{Z}) | ~ \rho(\tau)\in \partial \Sigma(\tau,\chi(\tau))~ a.e.~ \tau \in \mathfrak{J}, \ \chi\in \mathscr{L}^2(\mathfrak{J},\mathfrak{Z})\}.
\end{align*}
\begin{lemma}\label{lemma3.1}\cite{4}
 The set $S_{\Sigma,\chi}$ is non empty, and has convex, weakly compact values for each $\rho \in \mathscr{L}^2(\mathfrak{J},\mathfrak{Z})$ provided the assumption $(S_3)$ hold.
\end{lemma}
{\begin{lemma}\label{lemma3.2} \cite{LO}
Let the interval $[0,\beta]$ be  compact, and the set-valued map $\Sigma$ satisfy $(S_3)$. Let $F$ be a linear continuous operator from $\mathscr{L}^2([0,\beta],\mathfrak{Z})$ to $C([0,\beta],\mathfrak{Z})$. Then,
\begin{align*}
F \circ S_\Sigma: C([0,\beta],\mathfrak{Z}) \rightarrow P_{cp,cv}(\mathfrak{Z}),\ \chi \rightarrow (F \circ S_\Sigma) (\chi):=F(S_\Sigma,\chi)
\end{align*} has closed graph in $C([0,\beta],\mathfrak{Z}) \times C([0,\beta],\mathfrak{Z})$.
\end{lemma}}

Consider the space $\mathscr{W}_\beta = \{ \chi:]-\infty,\beta] \rightarrow \mathfrak{Z}:\ \chi_0=\eta \in \mathscr{W},\ \xi \in\mathfrak{Z},\  \chi|_{\mathfrak{J}}\in PC\big(\mathfrak{Z}\big),\ \mathop{\sup}\limits_{\tau\in \mathfrak{J}}\mathbb{E}\|\chi(\tau)\|^2 < \infty \}$ with the semi-norm
%\begin{align*}
$\|\chi\|_\beta= {\|\chi_0 \|}_\mathscr{W}+\big(\mathop{\sup}\limits_{\tau\in \mathfrak{J}}\mathbb{E}\|\chi(\tau)\|^2\big)^{1/2}.$\\
%\end{align*}

In view of Lemma~\ref{lemma2.1}, we have
\begin{align}\label{3.1}
\|y_\tau+\overline{\eta}_\tau\|_{\mathscr{W}}^2 \  \leq \ 2(\|y_\tau\|_{\mathscr{W}}^2+\|\overline{\eta}_\tau\|_{\mathscr{W}}^2) \leq& 4\{K_\beta^2 \mathop{\sup}\limits_{\tau\in \mathfrak{J}}\mathbb{E}\|y(s)\|^2+L_\beta^2 \|y_0 \|_{\mathscr{W}}^2+K_\beta^2 \mathop{\sup}\limits_{\tau\in \mathfrak{J}}\mathbb{E}\|\overline\eta(s)\|^2+L_\beta^2 \| \overline\eta_0 \|_\mathscr{W}^2\}\nonumber \\ \leq& 4\{K_\beta^2 \mathop{\sup}\limits_{\tau\in \mathfrak{J}}\mathbb{E}\|y(s)\|^2+L_\beta^2 \|\eta\|_\mathscr{W}^2\}, \ \tau\in \mathfrak{J},
\end{align}

Consider the set-valued map $\Theta:\mathscr{W}_\beta \rightarrow P(\mathscr{W}_\beta)$ characterized by $\Theta \chi$,
the set of all $\sigma \in \mathscr{W}_\beta $ satisfying
\begin{equation}
\sigma(\tau)=
\begin{cases}
 \eta(\tau),\ \hspace{9.8cm} \tau \in \mathfrak{J}_0; \\
 \mathcal{E}(\tau,0) \eta(0)+\mathcal{G}(\tau,0)\xi + \int_0^\tau \mathcal{G} (\tau,s) \rho(s)ds+\int_0^\tau \mathcal{G} (\tau,s)q(s,\chi_s)dZ_H(s),\hspace{.7cm} \tau \in [0,r_1];\\
 f_k(\tau, \chi_\tau),\hspace{8.8cm} \tau\in \mathop{\cup}\limits_{k=1}^{\mathcal{M}}(r_k,t_k]; \\
 \mathcal{E}(\tau,t_k)f_k(t_k, \chi_{t_k}) +\mathcal{G}(\tau,t_k)g_k(t_k, \chi_{t_k})
+ \int_{t_k}^\tau  \mathcal{G}(\tau,s) \rho(s)ds \\+ \int_{t_k}^\tau \mathcal{G} (\tau,s) q(s,\chi_s)dZ_H(s),\hspace{6.2cm} \tau\in \mathop{\cup}\limits_{k=1}^{\mathcal{M}}(t_k,r_{k+1}],
\end{cases}
\end{equation}
where $\rho \in S_{\Sigma,\chi}$.
We shall show that $\Theta$ has a fixed point in $\mathscr{W}_\beta$ that is a required solution for the system (\ref{mainequation}).

Define $\overline{\eta}(\cdot) : ]-\infty,\beta] \rightarrow \mathfrak{Z}$  by
\begin{equation*}
\overline{\eta}(\tau)=
\begin{cases}
\eta(\tau) ,\ \  \tau\in \mathfrak{J}_0;\\
0,\hspace{.8cm} \tau \in \mathfrak{J}.
\end{cases}
\end{equation*}
Obviously, $\overline{\eta}\in \mathscr{W}_\beta$ and $\overline{\eta}_0=\eta$. Set  $\chi(\tau)=\overline{\eta}(\tau)+y(\tau),$ $-\infty<\tau\leq \beta$.
Clearly $\chi(\cdot)$ satisfies \eqref{mildsolution} if and only if $y_0=0$ and
\begin{equation}
y(\tau)=
\begin{cases}
\mathcal{E}(\tau,0) \eta(0)+\mathcal{G}(\tau,0)\xi + \int_0^\tau \mathcal{G} (\tau,s) \rho(s)ds+\int_0^\tau \mathcal{G} (\tau,s)q(s,y_s+\overline{\eta}_s)dZ_H(s),\hspace{.7cm} \tau \in [0,r_1];\\
 f_k(\tau,y_\tau+\overline{\eta}_\tau),\hspace{8.8cm} \tau\in \mathop{\cup}\limits_{k=1}^{\mathcal{M}}(r_k,t_k]; \\
 \mathcal{E}(\tau,t_k)f_k(t_k, y_{t_k}+\overline{\eta}_{t_k}) +\mathcal{G}(\tau,t_k)g_k(t_k,  y_{t_k}+\overline{\eta}_{t_k})
+ \int_{t_k}^\tau  \mathcal{G}(\tau,s) \rho(s)ds \\+ \int_{t_k}^\tau \mathcal{G} (\tau,s) q(s,y_s+\overline{\eta}_s)dZ_H(s),\hspace{6.2cm} \tau\in \mathop{\cup}\limits_{k=1}^{\mathcal{M}}(t_k,r_{k+1}],
\end{cases}
\end{equation}

Consider the set $\mathscr{W}_\beta^0=\{\chi \in \mathscr{W}_\beta:\ y_0=0 \in \mathscr{W}\}$ with the semi-norm given by
$$\|y\|_\beta= {\|y_0 \|}_{\mathscr{W}}+\big(\mathop{\sup}\limits_{s\in \mathfrak{J}}\mathbb{E}\|y(s)\|^2\ \big)^{1/2}= \big(\mathop{\sup}\limits_{s\in \mathfrak{J}}\mathbb{E}\|y(s)\|^2\ \big)^{1/2}.$$ Then $(\mathscr{W}_\beta^0,\|\cdot\|_\beta)$ forms a Banach space.

Now suppose that the set-valued map $\bar\Theta:\mathscr{W}_\beta^0 \rightarrow P(\mathscr{W}_\beta^0)$ defined by $\bar\Theta y$,
the set of all $\bar\sigma \in \mathscr{W}_\beta^0 $ satisfying $\bar\sigma(\tau)=0,$ $ \tau \in \mathfrak{J}_0$ and
\begin{equation}
\bar\sigma(\tau)=
\begin{cases}
\mathcal{E}(\tau,0) \eta(0)+\mathcal{G}(\tau,0)\xi + \int_0^\tau \mathcal{G} (\tau,s) \rho(s)ds+\int_0^\tau \mathcal{G} (\tau,s)q(s,y_s+\overline{\eta}_s)dZ_H(s),\hspace{.7cm} \tau \in [0,r_1];\\
 f_k(\tau,y_\tau+\overline{\eta}_\tau),\hspace{8.8cm} \tau\in \mathop{\cup}\limits_{k=1}^{\mathcal{M}}(r_k,t_k]; \\
 \mathcal{E}(\tau,t_k)f_k(t_k, y_{t_k}+\overline{\eta}_{t_k}) +\mathcal{G}(\tau,t_k)g_k(t_k,  y_{t_k}+\overline{\eta}_{t_k})
+ \int_{t_k}^\tau  \mathcal{G}(\tau,s) \rho(s)ds \\+ \int_{t_k}^\tau \mathcal{G} (\tau,s) q(s,y_s+\overline{\eta}_s)dZ_H(s),\hspace{6.2cm} \tau\in \mathop{\cup}\limits_{k=1}^{\mathcal{M}}(t_k,r_{k+1}],
\end{cases}
\end{equation}
 where $\rho \in S_{\Sigma,y}=\{\rho\in L^2(\mathfrak{J},L(\mathfrak{Z},\mathscr{Z})\ | \ \rho(\tau)\in \partial \Sigma(\tau,y(\tau)+\overline{\eta}_\tau)\ a.e. \tau \in \mathfrak{J}\}$.
 If $\bar \Theta$ has a fixed point in $\mathscr{W}_\beta^0$ then $\Theta$ has a fixed point in $\mathscr{W}_\beta^0$.
We now assert that $\bar \Theta$ fulfils all assumptions of Lemma~\ref{lemma2.2}.
For $\varkappa>0$, let
%\begin{align*}
$D_\varkappa(0,\mathscr{W}_\beta^0)=\{y \in \mathscr{W}_\beta^0:\ \mathbb{E}\|y\|_\beta^2 \leq \varkappa\}$.
%\end{align*}
Clearly, $D_\varkappa \subset \mathscr{W}_\beta^0$ is convex, closed and bounded.
Now in view of inequality (\ref{3.1}) and Lemma~\ref{lemma2.1}, it follows that
\begin{align*}
\|y_\tau+\overline{\eta}_\tau\|_{\mathscr{W}}^2  \leq  4[K_\beta^2 \varkappa +L_\beta^2\|\eta\|^2_\mathscr{W}]=\varkappa^*,~ \tau\in \mathfrak{J}.
\end{align*}
Next, split ${\bar\Theta}={\bar\Theta}_1+{\bar\Theta}_2$, where
\begin{equation}
({\bar\Theta}_1 y)(\tau)=
\begin{cases}
 \mathcal{E}(\tau,0) \eta(0)+\mathcal{G}(\tau,0)\xi +\int_0^\tau \mathcal{G} (\tau,s)q(s,y_s+\overline{\eta}_s)dZ_H(s),\hspace{.7cm} \tau \in [0,r_1];\\
 f_k(\tau,y_\tau+\overline{\eta}_\tau),\hspace{6.6cm} \tau\in \mathop{\cup}\limits_{k=1}^{\mathcal{M}}(r_k,t_k]; \\
 \mathcal{E}(\tau,t_k)f_k(t_k, y_{t_k}+\overline{\eta}_{t_k}) +\mathcal{G}(\tau,t_k)g_k(t_k,  y_{t_k}+\overline{\eta}_{t_k})\\+ \int_{t_k}^\tau \mathcal{G} (\tau,s) q(s,y_s+\overline{\eta}_s)dZ_H(s),\hspace{4cm} \tau\in \mathop{\cup}\limits_{k=1}^{\mathcal{M}}(t_k,r_{k+1}],
\end{cases}
\end{equation}
and
\begin{equation}
({\bar\Theta}_2 y)(\tau)=
\begin{cases}
 \int_0^\tau \mathcal{G} (\tau,s) \rho(s)ds,\hspace{.7cm} \tau \in [0,r_1];\\
 0,\hspace{2.8cm} \tau\in \mathop{\cup}\limits_{k=1}^{\mathcal{M}}(r_k,t_k]; \\
  \int_{t_k}^\tau  \mathcal{G}(\tau,s) \rho(s)ds,\hspace{.7cm} \tau\in \mathop{\cup}\limits_{k=1}^{\mathcal{M}}(t_k,r_{k+1}].
\end{cases}
\end{equation}
\begin{lemma}\label{lm3.3}
 If $(S_1),\ (S_2)$ and $(S_4)$ hold, then ${\bar\Theta}_1$ takes bounded sets into bounded sets in $\mathscr W_\beta^0$, and is a contraction on $\mathscr{W}_\beta^0$.
\end{lemma}
\begin{proof}   \textbf{Claim 1: ${\bar\Theta}_1$ maps bounded sets to bounded sets in} $\mathscr W_\beta^0$.\\
 Let $ y \in D_\varkappa(0,\mathscr W_\beta^0)$, then by $(S_1),\ (S_2)$ and $(S_4)$, for $\tau \in [0,r_1]$, we get,
 \begin{align*}
\mathbb{E}\|y(\tau)\|^2_\mathfrak{Z} \ \leq
&\ 3M\big[\mathbb{E}\|\eta\|^2_\mathfrak{Z}+\mathbb{E}\|\xi \|^2_\mathfrak{Z}+ r_1^{2H}c(H)Tr(Q) \int_0^{r_1} m(s)m_q (\|y_s+\overline{\eta}_s\|^2_{\mathscr W})ds\big] \\ \leq
&\ 3M\big[\mathbb{E}\|\eta\|^2_\mathfrak{Z}+\mathbb{E}\|\xi \|^2_\mathfrak{Z}+ r_1^{2H}c(H)Tr(Q) m_q(\varkappa^*)\|m(\tau)\|_{L^1}:=s_0.
\end{align*}
For any $\tau \in (r_k,t_k],\ k=1,2,\hdots,\mathcal{M},$ 
\begin{align*}
\mathbb{E}\|y(\tau)\|^2_\mathfrak{Z}\   \leq& \  \mathbb{E}\| f_k(\tau,y_\tau+\overline{\eta}_\tau)\|^2_\mathfrak{Z}\  
 \leq \  \gamma_k (\|y_\tau+\overline{\eta}_\tau\|^2_\mathscr W+1) \leq \  \gamma_k (\varkappa^*+1):= \zeta_k.
\end{align*}
Similarly, for $\tau \in (t_k,r_{k+1}],\ k=1,2,\hdots,\mathcal{M}$, compute
\begin{align*}
\mathbb{E}\|y(\tau)\|^2_\mathfrak{Z}  \leq&\ 3[\mathbb{E}\| \mathcal{E}(\tau,t_k)f_k(t_k,y_{t_k}+\overline{\eta}_{t_k})\|^2_\mathfrak{Z}+\mathbb{E}\| \mathcal{G}(\tau,t_k)g_k(t_k,y_{t_k}+\overline{\eta}_{t_k})\|^2_\mathfrak{Z}\\
 &\ +\mathbb{E}\| \int_{t_k}^\tau \mathcal{G}(\tau,s)q(s,y_s+\overline{\eta}_s)dZ_H(s)\|^2_\mathfrak{Z} \\
  \leq& \ 3 M\big\{(\gamma_k +c_k) (\|y_{t_k}+\overline{\eta}_{t_k}\|^2_\mathscr W+1)+(r_{k+1}-t_k)^{2H}c(H)M \ Tr(Q) \int_{t_k}^\tau m(s) m_q (\|y_s+\overline{\eta}_s\|^2_\mathscr W) ds \big\}  \\
  \leq& \ 3 M\big\{(\gamma_k +c_k) (\varkappa^*+1)+ (r_{k+1}-t_k)^{2H}c(H)M \ Tr(Q) m_q(\varkappa^*)\|m(\tau)\|_{L^1}:= s_k.
\end{align*}
Set $\mathcal{N}= \mathop{\max}\limits_{0 \leq k \leq \mathcal{M}} \{s_k\} + \mathop{\max}\limits_{1 \leq k \leq \mathcal{M}}\{\zeta_k\} $, we get $\|{\bar{\Theta}}_1\|^2_\mathfrak{Z} \  \leq \ \mathcal{N}$.\\
  \textbf{Claim 2: $\bar \Theta_1$ is a contraction on} $\mathscr W_\beta^0$.\\
 Let $\chi^*,~\chi^{**}\in \mathscr{W}_\beta^0$. Then for $\tau \in [0,r_1]$, we have
\begin{align*}
 \mathbb{E}\|({\bar\Theta}_1\chi^*)(\tau)-({\bar\Theta}_1\chi^{**})(\tau)\|^2_\mathfrak{Z}=&\ \mathbb{E}\|\int_0^\tau \mathcal{G} (\tau,s)[q(s,\chi^*_s+\overline{\eta}_s)-q(s,\chi^{**}_s+\overline{\eta}_s)]dZ_H(s)\|^2_\mathfrak{Z}.
 \end{align*}
 Using Lemma~\ref{lemma2.1}, \ref{lemma2.3},   and  $(S_4)(i)$, we obtain
\begin{align*}
 \mathbb{E}\|({\bar\Theta}_1\chi^*)(\tau)-({\bar\Theta}_1 \chi^{**})(\tau)\|^2_\mathfrak{Z}\leq& \ c(H)M\tau^{2H} \mathbb{E}\|q(s,\chi^*_s+\overline{\eta}_s)-q(s,\chi^{**}_s+\overline{\eta}_s)\|^2_{L_2^0} \\ \leq& \ c(H)M\beta^{2H}Tr(Q)M_q\|\chi^*_s-\chi^{**}_s\|^2_\mathscr W \\ \leq& \ 2K_\beta^2c(H)M\beta^{2H}Tr(Q)M_q \mathop{\sup}\limits_{s\in \mathfrak{J}}\mathbb{E}\|\chi^*(s)-\chi^{**}(s)\|^2_\mathbb {Y}\\=&\  2K_\beta^2c(H)M\beta^{2H}Tr(Q)M_q \|\chi^*-\chi^{**}\|^2_{PC}
\end{align*}
Further, for $\tau \in \mathop{\cup}\limits_{i=1}^{\mathcal{M}}(r_k,t_k]$, using Lemma \ref{lemma2.1} and $(S_2)$(i), we have
\begin{align*}
\mathbb{E}\|({\bar\Theta}_1\chi^*)(\tau)-({\bar\Theta}_1 \chi^{**})(\tau)\|^2_\mathfrak{Z}\leq &\ \mathbb{E}\|f_k(\tau,\chi^*_\tau+\overline{\eta}_\tau)-f_k(\tau,\chi^{**}_\tau+\overline{\eta}_\tau)\|^2_\mathbb {Y}\\ \leq &\ \gamma_k \|\chi^*_\tau-\chi^{**}_\tau\|^2_\mathscr {W}\\ \leq &\ 4\gamma_k K_\beta^2 \mathop{\sup}\limits_{s\in \mathfrak{J}}\mathbb{E}\|\chi^*(s)-\chi^{**}(s)\|^2_\mathbb {Y}\\ \leq &\ 4\gamma_k K_\beta^2 \|\chi^*-\chi^{**}\|^2_{PC}.
\end{align*}
Lastly, for $\tau \in \mathop{\cup}\limits_{i=1}^{\mathcal{M}}(t_k,r_{k+1}]$,
\begin{align*}
\mathbb{E}\|({\bar\Theta}_1\chi^*)(\tau)-({\bar\Theta}_1 \chi^{**})(\tau)\|^2_\mathfrak{Z}\leq &\ 3\|\mathcal{E}(\tau,t_k)\|^2_\mathfrak{Z}\mathbb{E}\|f_k(t_k,\chi^*_{t_k}+\overline{\eta}_{t_k})-f_k(t_k,\chi^{**}_{t_k}+\overline{\eta}_{t_k})\|^2_\mathbb {Y}\\&+3 \|\mathcal{G}(\tau,t_k)\|^2_\mathfrak{Z}\|g_k(t_k,  \chi^*_{t_k}+\overline{\eta}_{t_k})-g_k(t_k,  \chi^*_{t_k}+\overline{\eta}_{t_k})\|^2_\mathfrak{Z}\\&+3 \mathbb{E}\| \int_{t_k}^\tau \mathcal{G} (\tau,s) [q(s,\chi^*_s+\overline{\eta}_s)-q(s,\chi^{**}_s+\overline{\eta}_s)]dZ_H(s)\|^2_\mathfrak{Z}\\ \leq &\ 3M\gamma_k \|\chi^*_{t_k}-\chi^{**}_{t_k}\|^2_\mathscr {W}+3Mc_k \|\chi^*_{t_k}-\chi^{**}_{t_k}\|^2_\mathscr {W}+3c(H)Tr(Q)MM_q \beta^{2H}\|\chi^*_s-\chi^{**}_s\|^2_\mathscr {W}\\ \leq &\ 12 MK_\beta^2[\gamma_k+c_k+\beta^{2H}c(H)Tr(Q)M_q]\|\chi^*-\chi^{**}\|^2_{PC}.
\end{align*}
Thus for $\tau \in \mathfrak{J}$,
\begin{align}\label{inequality}
\mathbb{E}\|({\bar\Theta}_1\chi^*)(\tau)-({\bar\Theta}_1 \chi^{**})(\tau)\|^2_\mathfrak{Z}\leq &\ M_0 \|\chi^*-\chi^{**}\|^2_{PC}.
\end{align}
\begin{align*}
\text{where} \ M_0= \ \mathop{\max}\limits_{1 \leq k \leq \mathcal{M}} 4K_\beta^2[(1+3M)\gamma_k+3M(c_k+\beta^{2H}c(H)Tr(Q)M_q)]<1.
\end{align*}
Hence $\bar \Theta_1$ is a contraction on $\mathscr W_\beta^0$.
\end{proof}
%\textbf{Step \rom{2}:}  $\bar \Theta_2$ \textbf {has compact convex values and it is completely continuous}.\\
%\begin{itemize}
 \begin{lemma}\label{lm3.4}
 If $(S_1)$ and $(S_3)$ hold, then $\bar \Theta_2$ has convex, compact values, and also is completely continuous.
 \end{lemma}
 \begin{proof}
\textbf{Claim 1: $\bar \Theta_2$ is convex for each} $\chi \in \mathscr W_\beta^0$.\\
If $\hat{\sigma}_1,~ \hat{\sigma}_2\in \bar \Theta_2 \chi$, then there are $\rho_1,~ \rho_2 \in S_{\Sigma,y}$ satisfying
 for any $\tau \in [t_k,r_{k+1}],\ k=0,1,\hdots,\mathcal{M}$
\begin{align*}
\hat{\sigma}_l(\tau)=   \int_{t_k}^\tau  \mathcal{G}(\tau,s) \rho_l(s)ds,~ l=1,2.
\end{align*}
Let $0\leq \lambda \leq 1$, then
$[\lambda \hat{\sigma}_1 +(1-\lambda) \hat{\sigma}_2](\tau)= \int_{t_k}^\tau  \mathcal{G}(\tau,s) [\lambda \rho_1(s)+(1-\lambda)\rho_2(s)] ds$.\\
In view of Lemma~\ref{lemma3.1}, $S_{\Sigma,y}$ is convex, we have
$\lambda \hat{\sigma}_1 +(1-\lambda) \hat{\sigma}_2 \in \bar \Theta_2 \chi$.\\
 \textbf{Claim 2: $\bar \Theta_2$ takes bounded sets into bounded sets in} $\mathscr W_\beta^0$.\\
 It is sufficient to establish that there exists $\mathcal{L}>0$ with the end goal that for each
 $\hat{\sigma}\in \bar \Theta_2 y,$ $ y \in D_\varkappa(0,\mathscr W_\beta^0)$, one has $\|\hat{\sigma}\|^2_{PC}\leq \mathcal{L}$.
 If $\hat{\sigma}\in {\bar \Theta}_2 y$, then there exists $\rho \in S_{\Sigma,y}$ for $\tau \in [t_k,r_{k+1}],\ k=0,1,\hdots,\mathcal{M}$ such that
 \begin{center}
 $\hat{\sigma}(\tau)= \int_{t_k}^\tau  \mathcal{G}(\tau,s) \rho(s)ds$.
  \end{center}
  Now, for $y \in D_\varkappa(0,\mathscr W_\beta^0)$,
% \end{itemize}
\begin{align*}
\mathbb{E}\|\hat{\sigma}(\tau)\|^2_{\mathbb Y} =\ \mathbb{E}\|\int_{t_k}^\tau  \mathcal{G}(\tau,s) \rho(s)ds\|^2_{\mathbb Y} \leq& \  \beta M \int_{t_k}^\tau\mathbb{E}\| \rho(s)\|^2_{\mathbb Y}ds\\ \leq& \ \beta M \int_{t_k}^\tau[b_1(s)+b_2\mathbb{E}\|y(s)+\bar \eta (s) \|^2_{\mathscr W}]ds \\ \leq& \ \beta M [\|b_1\|_{L^1(\mathfrak{J},\mathds{R}^+)}+2\beta b_2(\varkappa^*+\|\eta\|^2_{\mathscr{W}})]=\mathcal{L}.
\end{align*}
Thus for each $ y \in D_\varkappa(0,\mathscr W_\beta^0)$, we have
$ \|\hat{\sigma}\|^2_{PC}\leq \mathcal{L}.$

\noindent
\textbf{Claim 3: $\bar \Theta_2$ maps bounded sets into equicontinuous sets of} $\mathscr W_\beta^0$.\\
For every $y \in D_\varkappa(0,\mathscr W_\beta^0),$ $ \hat{\sigma} \in  \bar \Theta_2 y,$ there exists  $\rho \in S_{\Sigma,y}$
such that for $\tau \in [t_k,r_{k+1}],$ $ k=0,1,\hdots,\mathcal{M}$
\begin{center}
$\hat{\sigma}(\tau)=   \int_{t_k}^\tau  \mathcal{G}(\tau,s) \rho(s)ds.$
\end{center}
For $\tau,\ \tau+\varsigma \in [t_k,r_{k+1}],\ k=0,1,\hdots,\mathcal{M},\ 0<|\varsigma|<\delta,\ \delta>0$,
\begin{align*}
\mathbb{E}\|\hat{\sigma}(\tau+\varsigma)-\hat{\sigma}(\tau)\|^2_{\mathbb Y} \leq& \ 2\mathbb{E}\|\int_{t_k}^\tau  [\mathcal{G}(\tau+\varsigma,s)-  \mathcal{G}(\tau,s)] \rho(s)ds\|^2_{\mathbb Y}+2\mathbb{E}\|\int_{\tau}^{\tau+\varsigma} \mathcal{G}(\tau+\varsigma,s) \rho(s)ds\|^2_{\mathbb Y}\\ \leq & \  2 (\tau-t_k) \int_{t_k}^\tau \|\mathcal{G}(\tau+\varsigma,s)-  \mathcal{G}(\tau,s)\|^2_{\mathbb Y} \mathbb{E}\| \rho(s)\|^2_{\mathbb Y}ds+2 \tau \int_{\tau}^{\tau+\varsigma}\|\mathcal{G}(\tau+\varsigma,s)\|^2_{\mathbb Y}\mathbb{E}\| \rho(s)\|^2_{\mathbb Y}ds \\ \leq & 2(r_{k+1}-t_k)\{\|b_1\|_{L^1(\mathfrak{J},\mathds{R}^+)}+2(r_{k+1}-t_k)b_2(\varkappa^*+\|\eta\|^2_{\mathscr{W}})\}\\&\mathop{\sup}\limits_{s \in [t_k,r_{k+1}]}\|\mathcal{G}(\tau+\varsigma,s)-  \mathcal{G}(\tau,s)\|^2_{\mathbb Y} +2\tau M(\|b_1\|_{L^1(\mathfrak{J},\mathds{R}^+)}+b_2 \varkappa^* \tau).
\end{align*}
The compactness of the operator $\mathcal{G}(\tau,s)$ yields the continuity in the uniform operator topology.
Thus $\mathbb{E}\|\hat{\sigma}(\tau+\varsigma)-\hat{\sigma}(\tau)\|^2_{\mathbb Y} \rightarrow 0$ uniformly independently of $y \in D_\varkappa(0,\mathscr W_\beta^0)$ as $\varsigma \rightarrow 0$.
Hence our claim holds.
%Hence $\bar \Theta_2$ maps $D_\varkappa(0,\mathscr W_a^0)$ into an equicontinuous family of functions.\\
\textbf{Claim 4: $\bar \Theta_2$ is a compact set-valued map}.

\noindent
We now assert that $\bar \Theta_2$ maps $D_\varkappa(0,\mathscr W_\beta^0)$ into a precompact set in $\mathfrak{Z}$.
That is, the set $\triangle(\tau)=\{\hat{\sigma}(\tau),~ \hat{\sigma} \in  \bar \Theta_2 D_\varkappa(0,\mathscr W_\beta^0)\}$
 is relatively compact in $\mathfrak{Z}$. For $\tau=0,$ $ \bar \Theta_2 y=0$ is compact.  
 
 If $\tau\in [t_k,r_{k+1}],$ $ k=0,1,\cdots, n $ then for each $y \in D_\varkappa(0,\mathscr W_\beta^0)$ and $\hat{\sigma}(\tau)\in \bar \Theta_2 y,$ there exists $\rho \in S_{\Sigma,y}$ in order that
\begin{center} $\hat{\sigma}(\tau)=\int_{t_k}^\tau \mathcal{G}(\tau,s) \rho(s)ds$.\end{center}
Let $0<\epsilon<\tau$.
Define $\hat{\sigma}_{\epsilon}(\tau)= \int_{t_k}^{\tau-\epsilon}\mathcal{G}(\tau,s) \rho(s)ds= \mathcal{G}(\tau,\tau-\epsilon)\int_{t_k}^{\tau-\epsilon}\mathcal{G}(\tau-\epsilon,s) \rho(s)ds.$

By $(S_1)$, $\mathcal{G}(\tau,s);\ 0 < s \leq \tau  $ is compact. From the boundedness of $\int_{t_k}^{\tau-\epsilon}\mathcal{G}(\tau-\epsilon,s) \rho(s)ds$, we acquire that $\triangle_\epsilon (\tau)=\{\hat{\sigma}(\tau),\ \hat{\sigma} \in  \bar \Theta_2 D_\varkappa(0,\mathscr W_\beta^0)\}$ is relatively compact in $\mathfrak{Z}$.

Furthermore, for $\hat{\sigma}_{\epsilon} \in \bar \Theta_2 D_\varkappa(0,\mathscr W_\beta^0)\}$, we obtain
\begin{align*}
\mathbb{E}\|\hat{\sigma}(\tau)-\hat{\sigma}_\epsilon(\tau)\|^2
\leq& \ \epsilon\int_{\tau-\epsilon}^\tau \mathbb{E}\|\mathcal{G}(\tau,s)\rho(s)\|^2_{\mathbb Y}ds\\ \leq & \ \epsilon M \int_{\tau-\epsilon}^\tau [b_1(s)+b_2(\varkappa^*+\|\eta\|^2_{\mathscr W})]ds\\ \leq & \ \epsilon M[\epsilon\|b_1\|_{L^1(\mathfrak{J},\mathds{R}^+)}+\epsilon b_2 (\varkappa^*+\|\eta\|^2_{\mathscr W})]\\
\rightarrow &\ 0, ~\mbox{for sufficiently small positive}\  \epsilon.
\end{align*}
Thus, we have precompact sets that are arbitrarily close to $\triangle(\tau)$. Hence $\triangle(\tau),\ \tau>0$ is totally bounded.
Considering  \textbf{Claim 3} and the Arzela--Ascoli theorem, we infer that $\bar \Theta_2 $ is a compact operator (completely continuous).
 \end{proof}
 \begin{lemma}\label{lm3.5}
 If $(S_1)$ and $(S_3)$ hold, then  $\bar \Theta_2$ has a closed graph.
 \end{lemma}
 \begin{proof}
Let $w^{(n)} \rightarrow w^*,\ \hat{\sigma}_{(n)} \in \bar \Theta_2 w^{(n)},\ w^{(n)} \in D_\varkappa(0,\mathscr W_\beta^0)$ and $\hat{\sigma}_{(n)} \rightarrow  \hat{\sigma}_*$. Then using the Axiom (i), it follows that
\begin{align*}
 \|w^{(n)}_\tau -w^*_\tau \|^2_\mathscr W \leq & \ 2K_\beta^2 \mathop{\sup}\{\|w^{(n)}(s)-w^*(s) \|^2_\mathfrak{Z},\ 0\leq s \leq \tau\}+2L_\beta^2\|w^{(n)}_0 -w^*_0 \|^2_\mathscr W \\ \leq &\  2K_\beta^2 \mathop{\sup}\limits_{s \in \mathfrak{J}}\{\|w^{(n)}(s)-w^*(s) \|^2_\mathfrak{Z} \rightarrow \ 0 \ \mbox{as} \ n \rightarrow \infty.
\end{align*}
This implies $w^{(n)}_s \rightarrow w^*_s$ uniformly as $ n \rightarrow \infty$ for $s \in ]-\infty,\beta]$.

We claim that $\hat{\sigma}_* \in \bar \Theta_2 w^*.$
 For $\hat{\sigma}_{(n)} \in \bar \Theta_2 w^{(n)}$, there exists $\rho^{(n)}\in S_{\Sigma,w^{(n)}}$ such that, for $\tau \in [t_k,r_{k+1}],\ k=0,1,\hdots,\mathcal{M}$
\begin{align*}
\hat{\sigma}_{(n)}(\tau)=\int_{t_k}^\tau \mathcal{G}(\tau,s) \rho^{(n)}(s)ds.
\end{align*}
We wish to show that there is $\rho^* \in S_{\Sigma,w^*}$ that insures
\begin{align*}
\hat{\sigma}_*(\tau)=\int_{t_k}^\tau \mathcal{G}(\tau,s) \rho^*(s)ds,\ \tau \in [t_k,r_{k+1}],\ k=0,1,\hdots,\mathcal{M}.
\end{align*}
For any $\tau \in [t_k,r_{k+1}],\ k=0,1,2,\hdots \mathcal{M},$
\begin{align*}
\|\hat{\sigma}_{(n)}(\tau)-\hat{\sigma}_*(\tau)\|^2_{PC} =& \|\int_{t_k}^\tau \mathcal{G} (\tau,s)[\rho^{(n)}(s)-\rho^*(s)]ds\|^2_{PC} \\ \leq & \ (r_{k+1}-t_k)M \int_{t_k}^\tau \|\rho^{(n)}(s)-\rho^*(s)\|^2 \\ \rightarrow &\ 0 \ \mbox{as} \ n \rightarrow \infty.
\end{align*}
Consider the operator $\Delta:L^2([t_k,r_{k+1}],\mathfrak{Z}) \rightarrow C([t_k,r_{k+1}],\mathfrak{Z}),$ $k=0,1,\hdots, \mathcal{M}$,
\begin{align*}
\Delta(\rho)(\tau)=\int_{t_k}^\tau \mathcal{G}(\tau,s) \rho(s)ds.
\end{align*}
Then $ \|\Delta(\rho)\|^2  \leq \ (r_{k+1}-t_k)M \|\rho(s)\|^2.$
This shows that $\Delta$ is bounded, which implies $\Delta$ is continuous.
 Lemma \ref{lemma3.2} asserts that the operator $\Delta \circ S_\Sigma$ has a closed graph.
Moreover, by the definition of $\Delta$, for $\tau \in [t_k,r_{k+1}],\  k=0,1, \hdots, \mathcal{M},$ we get,
\begin{align*}
\hat{\sigma}_{(n)}(\tau) \in \bar{\Theta}_2(S_{\Sigma,w^{(n)}}).
\end{align*}
Since $w^{(n)}\rightarrow w^*,$ for some $\rho^* \in S_{\Sigma,w^*},$ it follows that for any $\tau \in [t_k,r_{k+1}],\ k=0,1,\hdots \mathcal{M},$ we have
\begin{align*}
\hat{\sigma}_*(\tau)=\int_{t_k}^\tau \mathcal{G}(\tau,s)\rho^*(s)ds,\  \mbox{this indicates that}\ \hat{\sigma}_* \in \bar{\Theta}w^*.
\end{align*}
This implies operator $\bar{\Theta}_2$ is closed graph.
 By utilizing Proposition 3.3.12(2) of \cite{SA}, we get that $\bar{\Theta}_2$ is
u.s.c.. Therefore $\bar{\Theta}_2$ satisfies the condition (b) of Lemma \ref{lemma2.2}.
\end{proof}

%\noindent
%\textbf{Step \rom{3}:  $\bar{\Theta}$ has a solution in $\mathscr{W}_\beta^0.$}
%
%\noindent
\begin{theorem}\label{thm3.1}
Let $(S_1)$-$(S_4)$ are fulfilled.  Then system~\eqref{mainequation} admits at least one solution on $]-\infty,\beta]$ provided that
\begin{align}\label{3.2}
M_0=\mathop{\max}\limits_{1\leq k \leq \mathcal{M}}4K_\beta^2[\gamma_k+3M(c_k+\beta^{2H}c(H)Tr(Q)M_q)]<1,\ \text{and} \nonumber \\
\int_0^\beta \max \{c_1^* m(t)+c_2^*b_2\}dt  <   \int_{\Upsilon(0)}^\infty \frac{ds}{m_q(s)} .
\end{align}
\end{theorem}
\begin{proof}
We claim that the set $\mho=\{y \in \mathscr{W}_\beta^0: \lambda y \in  \bar{\Theta}y= \bar{\Theta}_1 y+ \bar{\Theta}_2y\}$ is bounded for some $\lambda > 1$ on $[0,\beta]$. Let $y \in \mathscr{W}_\beta^0$ satisfies $\lambda y \in \bar{\Theta}y= \bar{\Theta}_1 y+ \bar{\Theta}_2y$ for some $\lambda >1$  , we obtain
\begin{equation}
y(\tau)=
\begin{cases}
\frac{1}{\lambda}\Big[\mathcal{E}(\tau,0) \eta(0)+\mathcal{G}(\tau,0)\xi + \int_0^\tau \mathcal{G} (\tau,s) \rho(s)ds+\int_0^\tau \mathcal{G} (\tau,s)q(s,y_s+\overline{\eta}_s)dZ_H(s)\Big],\hspace{.7cm} \tau \in [0,r_1];\\
 \frac{1}{\lambda}f_k(\tau,y_\tau+\overline{\eta}_\tau),\hspace{8.8cm} \tau\in \mathop{\cup}\limits_{k=1}^{\mathcal{M}}(r_k,t_k]; \\
 \frac{1}{\lambda}\Big[\mathcal{E}(\tau,t_k)f_k(t_k, y_{t_k}+\overline{\eta}_{t_k}) +\mathcal{G}(\tau,t_k)g_k(t_k,  y_{t_k}+\overline{\eta}_{t_k})
+ \int_{t_k}^\tau  \mathcal{G}(\tau,s) \rho(s)ds \\+ \int_{t_k}^\tau \mathcal{G} (\tau,s) q(s,y_s+\overline{\eta}_s)dZ_H(s)\Big],\hspace{6.2cm} \tau\in \mathop{\cup}\limits_{k=1}^{\mathcal{M}}(t_k,r_{k+1}],
\end{cases}
\end{equation}
Thus for $\tau \in [0,r_1]$, we get
\begin{align*}
\mathbb{E}\|y(\tau)\|^2_\mathfrak{Z} \ \leq
 %&\ 4\big\{ \mathbb{E}\|\mathcal{E}(\tau,0)\eta(0)\|^2_\mathfrak{Z}+\mathbb{E}\|\mathcal{G}(\tau,0)\xi \|^2_\mathfrak{Z}+
 %\mathbb{E}\|\int_0^\tau \mathcal{G}(\tau,s)\rho(s)ds\|^2_\mathfrak{Z}+ \mathbb{E}\|\int_0^\tau \mathcal{G}(\tau,s)q(s,y_s
 %+\overline{\eta}_s)ds\|^2_\mathfrak{Z}\big\} \\ \  \leq&\ 4M \big\{ \mathbb{E}\|\eta(0)\|^2_\mathfrak{Z}+\mathbb{E}\|\xi \|^2_\mathfrak{Z}
 %+r_1 \int_0^{r_1}[b_1(s)+b_2 \mathbb{E}\|y(s)+\overline{\eta}(s)\|^2_\mathfrak{Z} ]ds
 %+r_1 \int_0^{r_1} \mathbb{E}\|q(s,y_s+\overline{\eta}_s)\|^2_{L_2^0}ds\big\} \\ \  \leq
&\ 4M\big[\mathbb{E}\|\eta\|^2_\mathfrak{Z}+\mathbb{E}\|\xi \|^2_\mathfrak{Z}+r_1 \int_0^{r_1} b_1(s)ds+r_1 b_2 \int_0^{r_1} \mathbb{E}\|y(s)\|^2_\mathfrak{Z} ds \\
\ &+ r_1^{2H}c(H)Tr(Q) \int_0^{r_1} m(s)m_q (\|y_s+\overline{\eta}_s\|^2_{\mathscr W})ds\big].
\end{align*}
For any $\tau \in (r_k,t_k],\ k=1,2,\hdots,\mathcal{M},$ we have
\begin{align*}
\mathbb{E}\|y(\tau)\|^2_\mathfrak{Z}  \leq&\ \mathbb{E}\| f_k(\tau,y_\tau+\overline{\eta}_\tau)\|^2_\mathfrak{Z} \\
 \leq& \ \gamma_k (\|y_\tau+\overline{\eta}_\tau\|^2_\mathscr W+1).
\end{align*}
Similarly, for $\tau \in (t_k,r_{k+1}],\ k=1,2,\hdots,\mathcal{M}$, we compute
\begin{align*}
\mathbb{E}\|y(\tau)\|^2_\mathfrak{Z}  \leq&\ 4[\mathbb{E}\| \mathcal{E}(\tau,t_k)f_k(t_k,y_{t_k}+\overline{\eta}_{t_k})\|^2_\mathfrak{Z}+\mathbb{E}\| \mathcal{G}(\tau,t_k)g_k(t_k,y_{t_k}+\overline{\eta}_{t_k})\|^2_\mathfrak{Z}\\
 &\ +\mathbb{E}\| \int_{t_k}^\tau \mathcal{G}(\tau,s)q(s,y_s+\overline{\eta}_s)dZ_H(s)\|^2_\mathfrak{Z} +\mathbb{E}\| \int_{t_k}^\tau \mathcal{G}(\tau,s)\rho(s)ds\|^2_\mathfrak{Z} \\
  \leq& \ 4 M\big\{(\gamma_k +c_k) (\|y_{t_k}+\overline{\eta}_{t_k}\|^2_\mathscr W+1)+(r_{k+1}-t_k)^{2H}c(H)M \ Tr(Q) \int_{t_k}^\tau m(s) m_q (\|y_s+\overline{\eta}_s\|^2_\mathscr W) ds \\
 &+(r_{k+1}-t_k)M \int_{t_k}^\tau[b_1(s)+b_2 \mathbb{E}\|y(s)+\overline{\eta}(s)\|^2_\mathfrak{Z}]ds\big\}
\end{align*}
Using Lemma \ref{lemma2.1}, we get
$\sup \{\|y_s+\overline{\eta}_s\|_{\mathscr{W}}^2,\ \tau \geq s \geq 0 \} \  \leq \ 4  L_\beta^2 \mathbb{E}\|\eta\|_{\mathscr{W}}^2 + 4 K_\beta^2 \sup \{\mathbb{E}\|y(s)\|^2 ,\ \tau \geq s \geq 0  \}$.

Let $\varphi(\tau)= 4  L_\beta^2 \mathbb{E}\|\eta\|_{\mathscr{W}}^2 + 4 K_\beta^2 \sup \{\mathbb{E}\|y(s)\|^2 ,\  \tau \geq s \geq 0 \},\ \tau \geq 0.$
Thus, for $\tau \in \mathfrak{J}$, we obtain
\begin{align*}
\mathbb{E}\|y(\tau)\|^2_\mathfrak{Z} \leq & \hat{M}+\gamma_k \varphi(\tau)+4M[\gamma_k \varphi (\tau)+c_k \varphi (\tau)]+4\beta^{2H} c(H)M\  Tr(Q) \int_0^\tau m(s)m_q (\varphi (s)) ds\\
&  + \beta M \int _0^\tau [b_1(s)+b_2 \varphi (s)]ds,
\end{align*}
where  $\widehat{M}= \mathop {\max}\limits_{1 \leq k \leq \mathcal{M}} [4M (\|\eta\|^2_\mathscr W + \mathbb{E}\|\xi\|^2_\mathfrak{Z}+\gamma_k +c_k)+\gamma_k].$

Also a simple calculation yields that
\begin{align*}
\varphi(\tau) \leq &  4[ L_\beta^2 \|\eta\|^2 + K_\beta^2 \hat{M}]+ 4K_\beta^2 \big\{[\gamma_k(1+4M)+4Mc_k]\varphi(\tau)\\
&+4\beta^{2H}c(H)M \ Tr(Q) \int_0^\tau m(s) m_q(\varphi(s)) ds + \beta M \int_0^\tau [b_1(s)+b_2 \varphi (s)]ds \big\}.
\end{align*}
Using the fact that $\widehat{M}_0= 4K_\beta^2 \mathop {\max}\limits_{1 \leq k \leq \mathcal{M}}\{\gamma_k (1+4M)+4Mc_k\}<1$, we obtain
\begin{align*}
\varphi(\tau) \leq \frac{1}{1-\widehat{M}_0}\{4 L_\beta^2 \|\eta\|^2_\mathscr W+4K_\beta^2 \widehat{M}\}+c_1^* \int_0^\tau m(s)m_q(\varphi(s)) ds + c_2^*\  b_2 \int_0^\tau  \varphi(s) ds,
\end{align*}
where  $c_1^* = \frac{4K_\beta^2 \beta^{2H} c(H) M \ Tr(Q)}{1-\widehat{M}_0},$                                           $ c_2^*= \frac{\beta M}{1-\widehat{M}_0}.$

Let $ c^*= \frac{1}{1-\widehat{M}_0}\{4 L_\beta^2 \|\eta\|^2_\mathscr W +4K_\beta^2 \widehat{M}+\beta M \|b_1\|_{L_1(\mathfrak{J},\mathds{R}^+)}\}.$
Then the above inequality can be rewritten as
\begin{align*}
\varphi (\tau) \leq \Upsilon (\tau) =c^*+c_1^* \int_0^\tau m(s)m_q(\varphi(s)) ds+ c_2^* b_2 \int_0^\tau \varphi(s)ds,
\end{align*}
Also $\Upsilon(0)=c^*$ and
\begin{align*}
\varphi'(\tau) \leq \  c_1^* m(t)m_q(\varphi(\tau))+c_2^*b_2 \varphi(\tau) \leq \  \max \{c_1^*m(\tau),\ c_2^* b_2\}[\Upsilon(\tau)+m_q(\Upsilon(\tau))],\ \tau \in \mathfrak{J}.
\end{align*}
Thus we get
\begin{align*}
\int_0^\tau \frac{\Upsilon'(s)}{\Upsilon(s)+m_q(\Upsilon(s))}ds\ \leq \ \int_0^\tau \frac{\Upsilon'(s)}{m_q(\Upsilon(s))} ds\ \leq \ \int_0^\beta \max \{c_1^* m(s)+c_2^*b_2\}ds.
\end{align*}
Moreover,
\begin{align*}
\int_{\Upsilon(0)}^{\Upsilon(\tau)} \frac{ds}{m_q(s)}  \leq   \int_0^\beta \max \{c_1^* m(s)+c_2^*b_2\}ds  <   \int_{\Upsilon(0)}^\infty \frac{ds}{m_q(s)}.
\end{align*}
The above inequality shows that $\Upsilon(\tau)$ is bounded. Therefore, we have $\widetilde{N}$ such that
\begin{align*}
\Upsilon(\tau)\leq \widetilde{N}, ~\tau \in \mathfrak{J}.
\end{align*}
Consequently, $\|y_\tau+\eta_\tau\|^2_\mathscr W  \leq \varphi(\tau) \leq \Upsilon(\tau)  \leq \widetilde{N},$ $ \tau \in \mathfrak{J}$,
where $\widetilde{N}$ depends on $m_q(\cdot)$ and $m(\cdot)$.
This proves that $\mho$ is bounded on $[0,\beta]$. Hence, Lemmas \ref{lm3.3}--\ref{lm3.5} and first assertion of Lemma \ref{lemma2.2} yield that
$\bar{\Theta}=\bar{\Theta}_1+\bar{\Theta}_2$ has a fixed element $y^*$ in $\mathscr{W}_\beta^0.$
Set $\chi^*(\tau)=y^* (\tau)+\bar{\eta}(\tau),$ $ \tau \in ]-\infty,\beta].$
Then $\chi^*$ is a fixed point of the operator $\Theta$. Consequently, $\chi^*$ is a mild solution of the system (\ref{mainequation}).
\end{proof}
\section{An Example}\label{Sect.4}
This section is illustrated for the applicability of the above result to a concrete stochastic partial differential inclusions
 with unbounded delay and Clarke's subdifferential given by
\begin{equation}
\begin{cases}
\frac{\partial^2}{\partial \tau^2}\chi(\tau,w) \in \frac{\partial^2}{\partial w^2}\chi(\tau,w)+v(\tau)\frac{\partial}{\partial \tau}\chi(\tau,w)+\partial \Sigma(\tau,w,\chi(\tau,w))+ \int_{-\infty}^{\tau}u(t-\tau)\tilde{u}(\tau,\chi(t,w))dZ_H(t),\\ \qquad (\tau,w)\in \mathop{\cup}\limits_{k=1}^\mathcal{M} (t_k,r_{k+1}] \times [0,\pi]; \\
 \chi(\tau,w)=\int_{-\infty}^\tau \mu_k(t-\tau)\chi(t,w)dt,\quad (\tau,w)\in \mathop{\cup}\limits_{k=1}^\mathcal{M}(r_k,t_k]\times [0,\pi];
\\
\chi(\tau,0)=\chi(\tau,\pi)=0 , \ \tau\in (0,\beta];\\
 \chi(\tau,w)=\eta(\tau,w)\in \mathscr{W},\ (\tau,w)\in ]-\infty,0]\times[0,\pi];\\
 \frac{\partial}{\partial \tau}\chi(0,w)= \chi_1(w);\\
 \frac{\partial}{\partial \tau}\chi(\tau,w)= \int_{-\infty}^\tau \tilde{\mu}_k(t-\tau)\chi(t,w)dt,\quad (\tau,w)\in \mathop{\cup}\limits_{k=1}^\mathcal{M}(r_k,t_k]\times [0,\pi];
\end{cases}
\end{equation}
where $\eta, \chi_1$ are continuous. To compose the above system in the abstract form, set $\mathfrak{Z}=\mathscr{Z}=L^2([0,\pi];\mathds{R})$ 
 Let $\mathscr{H}^2([0,\pi],\mathds{R})$ be the Sobolev space of all mappings $\chi:[0,\pi] \rightarrow \mathds{R}$ such that $\chi'' \in L^2([0,\pi],\mathds{R})$. 
 Define $\mathcal{A}:D(\mathcal{A}) \rightarrow \mathfrak{Z}$ by $\mathcal{A}\chi(\tau)=\chi''(\tau)$, where $D(\mathcal{A})=\{\chi \in \mathfrak{Z}: \chi,\  chi'$ are absolutely continuous, $\chi'' \in \mathfrak{Z},\ \chi(0)=\chi(\pi)=0\}$. Then, the cosine family $C(\tau)$ and the associated sine function $S(\tau)$ on $\mathfrak{Z}$ are generated by $\mathcal{A}$ and are strongly continuous; also for any $\tau \in \mathds{R},$ $ \|C(\tau)\| \leq 1$ \cite{TW}).
 Define $\widehat{\mathcal{C}}: \mathscr{H}^1([0,\pi],\mathds{R}) \rightarrow \mathfrak{Z}$ by $\widehat{\mathcal{C}}(\tau)\chi(w)=v(\tau)\chi'(w)$, where $v:[0,1] \rightarrow \mathds{R}$ is H\"{o}lder continuous. Define the linear operator $\mathcal{A}(\tau)= \widehat{\mathcal{C}}(\tau)+\mathcal{A}$ that is  closed also. The operator $\{\mathcal{A}(\tau): \tau \in \mathfrak{J}\}$ generates the evolution operator $ \{\mathcal{G}(\tau,s)\}_{(\tau,s) \in D},\ D=\{(\tau,s)\in \mathfrak{J} \times \mathfrak{J}: s \leq \tau\}$, see \cite{Henriquez2011}. Moreover, $\mathcal{G(\cdot,\cdot)}$ is well defined and assumption $(S_1)$ hold with $M=1$.
 
The map $\Sigma: [0,\pi]\times \mathfrak{J}\times \mathds{R} \rightarrow \mathds{R}$ is a locally Lipschitz w.r.t. the last variable which is non-smooth and non-convex. The set-valued function $\partial \Sigma (\tau,w,\phi): \mathds{R} \rightarrow 2^\mathds{R}$ is non-monotone. 
To support $(S_3)$ one can take $\Sigma(\phi)= \min \{\varpi_1(\phi),\varpi_2(\phi)\}$, where $\varpi_1,\varpi_2:\mathds{R}\rightarrow \mathds{R}$ are convex quadratic functions~\cite{SM}. Notation $Z_H(\tau)$ stands for the Rosenblatt process that is defined on the complete stochastic space $(\Omega,\Gamma,\mathbb{P})$ and $\frac{1}{2} < H < 1$ .

 Let the function $\tilde{l}: ]-\infty,0]\rightarrow \mathds{R}^+\cup\{0\}$ be measurable satisfying (g-5)-(g-7) described in \cite{Hino}. 
 
Set ${PC}_{0} \times  L^2(\widetilde{l},\mathfrak{Z}) =\big\{\Pi:\mathfrak{J}_0\rightarrow \mathfrak{Z},\ \Pi(\cdot) $ is Lebesgue measurable on $]-\infty,0)\big\}$ and
\begin{align*}
\|\Pi\|_\mathscr{W}= \|\Pi(0)\|+\Big(\int_{-\infty}^0 \tilde{l}(s)\|\Pi(s)\|^2 ds\Big)^\frac{1}{2}.
\end{align*}
The space $(\mathscr{W},\|\cdot\|_\mathscr{W})= ({PC}_0 \times L^2(\tilde{l},\mathfrak{Z}),\|\cdot\|_\mathscr{W})$ satisfies Axioms (i) and (ii), (see \cite{Hino}).
% with $J=1,\  K(\tau)=1+(\int_{-\tau}^0\tilde{l}(s)ds)^{\frac{1}{2}}, \tau\geq 0$, and $L(\tau)= G(-\tau)^{\frac{1}{2}}$, where G is as defined in $(g$-$6)$.

Suppose that the following conditions hold:
\begin{enumerate}
\item[(i)] Let $u:\mathds{R} \rightarrow \mathds{R},$ $ \tilde{u}:\mathds{R}^2 \rightarrow \mathds{R}$ be continuous and $L_u=\big(\int_{-\infty}^0 \frac{(u(s))^2}{\tilde{l}(s)}ds\big)^{1/2}< \infty$, also  for $(\tau,x)\in \mathds{R}^2$, 
    $|\tilde{u}(\tau,x)| \leq \tilde{b}(\tau)|x|$, $\tilde{b}:\mathds{R} \rightarrow \mathds R$ is continuous.
\item[(ii)] The functions $\mu_k, \tilde{\mu}_k:\mathds{R}^2 \rightarrow \mathds R$ are continuous and there are  mappings $a_k,\tilde{a}_k :\mathds R \rightarrow \mathds R $ which are continuous satisfying $|\mu_k(s,x)| \leq a_k(s)$ with $\mathfrak{A}_k= \big(\int_{-\infty}^0 \frac{(a_k(s))^2}{\tilde{l}(s)}ds\big)^{1/2}< \infty$, also $|\tilde{\mu}_k(s,x)| \leq \tilde{a}_k(s)$ with $\tilde{\mathfrak{A}}_k= \big(\int_{-\infty}^0 \frac{(\tilde{a}_k(s))^2}{\tilde{l}(s)}ds\big)^{1/2}< \infty$.
\end{enumerate}
Take $\eta \in \mathscr{W}$ with $\eta(\vartheta)(w)= \eta(\vartheta,w),$ $ (\vartheta,w)\in ]-\infty,0]\times \mathscr{W}$.

 Let $\chi(t)(w)=\chi(t,w)$, define $q:\mathfrak{J} \times \mathscr{W} \rightarrow L_2^0,\ f_k,g_k (r_k,t_{k+1}] \times \mathscr{W} \rightarrow \mathfrak{Z} $ as
\begin{align*}
q(\tau,\Xi)(w)= \int_{-\infty}^{0}u(t)\tilde{u}(\tau,\Xi(t)(w)dt,\\
f_k(\tau,\Xi)(w)= \int_{-\infty}^{0}\mu_k(t)\Xi(t)(w)dt,\\
g_k(\tau,\Xi)(w)= \int_{-\infty}^{0}\tilde{\mu}_k(t)\Xi(t)(w)dt.
\end{align*}
Under the above assumptions the problem (\ref{Sect.4}) can be formulated as (\ref{mainequation}).

From the hypothesis (i), for all $(\tau, \Xi) \in [0,\beta) \times \mathscr{W}$, we have
\begin{align*}
\mathbb{E}\|q(\tau,\Xi)\|^2 &= \mathbb{E}\Big[\Big(\int_0^\pi\Big(\int_{-\infty}^{0}u(t)\tilde{u}(\tau,\Xi(t)(w))dt\Big)^2dw\Big)^{\frac{1}{2}}\Big]^2\\ & \leq \mathbb{E}\Big[\Big(\int_0^\pi\Big(\int_{-\infty}^{0}u(t)\tilde{b}(\tau)|\Xi(t)(w)|dt\Big)^2dw\Big)^{\frac{1}{2}}\Big]^2\\ & \leq \mathbb{E}\Big[\tilde{b}(\tau)\Big(\int_{-\infty}^{0}\frac{(u(t))^2}{\tilde{l}(t)}dt\Big)^{\frac{1}{2}}\Big(\int_{-\infty}^{0}\tilde{l}(t)\|\Xi(t)\|^2dt\Big)^\frac{1}{2}\Big]^2\\ & \leq [\tilde{b}(\tau) L_u]^2\|\Xi\|^2_\mathscr{W}.
\end{align*}
Also for all $ (\tau, \Xi),~ (\tau, \Xi_1) \in (r_k,t_k) \times \mathscr{W}$, we get
\begin{align*}
\mathbb{E}\|g_k(\tau,\Xi)-g_k(\tau,\Xi_1)\|^2 &= \mathbb{E}\Big[\Big(\int_0^\pi\Big(\int_{-\infty}^{0}\mu_i(t,w)[\Xi(t)(w)-\Xi_1(t)(w)]dt\Big)^2dw\Big)^{\frac{1}{2}}\Big]^2\\ & \leq \mathbb{E}\Big[\Big(\int_{-\infty}^{0}\frac{(a_k(t))^2}{\tilde{l}(t)}dt\Big)^{\frac{1}{2}}\Big(\int_{-\infty}^{0}\tilde{l}(t)\|\Xi(t)-\Xi_1(t)\|^2dt\Big)^\frac{1}{2}\Big]^2\\ & \leq \Big[\mathfrak{A}_k \Big(\|\Xi(0)\|+\Big(\int_{-\infty}^{0}\tilde{l}(t)\|\Xi(t)-\Xi_1(t)\|^2dt\Big)^\frac{1}{2}\Big)\Big]^2 \\ & \leq \mathfrak{A}_k^2 \|\Xi-\Xi_1\|^2_\mathscr{W}.
\end{align*}
Similarly,
\begin{align*}
\mathbb{E}\|f_k(\tau,\Xi)-f_k(\tau,\Xi_1)\|^2 & \leq  \gamma_k\|\Xi-\Xi_1\|^2_\mathscr{W},\ \gamma_k>0,\ \mbox{for all}\ (\tau, \Xi),\ (\tau, \Xi_1) \in (r_k,t_k) \times \mathscr{W} \\
\mathbb{E}\|q(\tau,\Xi)-q(\tau,\Xi_1)\|^2 & \leq M_q  \|\Xi-\Xi_1\|^2_\mathscr{W},\ M_q>0,\ \mbox{for all}\ (\tau, \Xi),\ (\tau, \Xi_1) \in [0,\beta) \times \mathscr{W}
\end{align*}
Thus all the hypotheses in Theorem \ref{thm3.1} are followed. Hence, the model (\ref{Sect.4}) admits a solution on $\mathfrak{J}.$

\section{Conclusion} 
 In this article, we study a new class of non-autonomous second-order stochastic inclusions of Clarke's subdifferential type involving NIIs, unbounded delay, and the Rosenblatt process. The existence result is deduced by utilizing the fixed point strategy for a set-valued map. The obtained results are illustrated through a concrete example. In the future, it is interesting to study the controllability results (such as approximate controllability, optimal control, time-optimal control among others) of the associated systems. In our future research work, we will consider the optimal control problem associated with the system (\ref{mainequation}) involving state-dependent delay.

\end{document}